\documentclass[12pt]{amsart}

\usepackage{amsfonts}
\usepackage{amsmath}
\usepackage{amssymb}
\usepackage{amsthm} 
\usepackage{mathrsfs}
\usepackage{xcolor}
\usepackage{appendix}
\usepackage[colorlinks=true,urlcolor=blue,linkcolor=blue,citecolor=black]{hyperref}

\newcommand\utimes{\mathbin{\ooalign{$\cup$\cr%
   \hfil\raise0.42ex\hbox{$\scriptscriptstyle\times$}\hfil\cr}}}
\newcommand\bigutimes{\mathop{\ooalign{$\bigcup$\cr%
   \hfil\raise0.36ex\hbox{$\scriptscriptstyle\boldsymbol{\times}$}\hfil\cr}}}
\usepackage[T1]{fontenc}
\usepackage{mathabx,graphicx}

\theoremstyle{definition}
\newtheorem{thm}{Theorem}[section] 
\newtheorem{cor}[thm]{Corollary}
\newtheorem{lem}[thm]{Lemma}
\newtheorem{prop}[thm]{Proposition}
\theoremstyle{definition}
\newtheorem{defn}[thm]{Definition}
\newtheorem{rem}[thm]{Remark}
\newtheorem{nota}[thm]{Notation}

\numberwithin{equation}{section} 

\author[PL. Tseng]{Pei-Lun Tseng}
\address{New York University Abu Dhabi, Division of Science, Mathematics, Abu Dhabi, UAE}
\email{pt2270@nyu.edu}

\date{\today}


\title{Operator-Valued Infinitesimal Multiplicative Convolutions}

\begin{document}

\maketitle

\begin{abstract}

We consider the notions of operator-valued infinitesimal (OVI) free independence, OVI Boolean independence, and OVI monotone independence. For each notion of OVI independence, we introduce the corresponding infinitesimal transforms, and then we show that the transforms satisfy certain multiplicative property. Additionally, we extend the concept of $t$-coefficients to the infinitesimal framework and investigate its properties. Finally, we present an application involving complex Wishart matrices utilizing our infinitesimal free multiplicative formula.

\end{abstract}

\section{Introduction}

In non-commutative probability theory, various concepts of independence have been explored. In \cite{Mur03}, Muraki demonstrated that there are only five types of independence, which are tensor, free \cite{voi85}, Boolean \cite{spe93}, monotone \cite{Mur2000}, and anti-monotone \cite{Mur2000} that exhibit specific universal properties.

The notion of free independence was introduced by Voiculescu in 1985 \cite{voi85}. It plays an important role in the studying of the asymptotic behavior of random matrices. To be precise, it provides us a way to study the distribution of the eigenvalues of polynomials in several large random matrices. Over time, numerous extensions and generalizations of free probability have emerged. Infinitesimal freeness is one of the generalization of freeness \cite{BGN},\cite{BS12}. 

An application of infinitesimal freeness to random matrix theory was first presented by Shlyakhtenko \cite{SH18}, and he showed how to understand the finite rank perturbation in some random matrix models by applying infinitesimal freenss. Subsequently, many works on infinitesimal free probability and its application to random matrices are studied \cite{coll18}, \cite{dall19}, \cite{mingo19}, and \cite{au21} . The infinitesimal analogous for Boolean independence and monotone independence were also investigated in \cite{CEKP19} and \cite{Has11}.   

Another extension of freeness is known as operator-valued (OV) freeness \cite{VOI95}, which parallels to the classical concept of conditional independence. This theory was applied by Belinschi, Mai, and Speicher to describe the limit behavior of polynomial functions of some random matrix models via a linearization trick \cite{BMN17}. In addition, OV analogues of Boolean independence was studied by Popa \cite{POP09}. Additionally, monotone independence within the OV framework has been studied in \cite{POP08} and \cite{Has14}.

Free convolution plays a crucial rule in free probability theory. Roughly speaking, if $x$ and $y$ are two random variables that are freely independent, the law of $x+y$ (respectively $xy$) is called the free additive (respectively free multiplicative) convolution of the laws of $x$ and $y$ \cite{voi85} \cite{voi87}. Analogous convolutions also exist for Boolean independent random variables \cite{spe93} \cite{FR04} and for monotone independent random variables \cite{Mur2000} \cite{Ber05}. The infinitesimal analogue of free convolution was introduced by Belinschi and Shlyakhtenko \cite{BS12}. Moreover, convolution for free random variables in the OV framework was presented by Voiculescu \cite{VOI95}, and the OV the analogue of convolution for Boolean and monotone random variables have also been studied in \cite{POP08} \cite{POP09}.    

Curran and Speicher introduced the notion of infinitesimal freeness in the OV framework, presenting a random matrix model which are asymptotically OVI freeness \cite{CS11}. A detailed discussion of OVI freeness is covered in \cite{PL19}. In particular, the formula of OVI free additive convolution was also provided. Furthermore, the notions of OVI Boolean and monotone independence were investigated in \cite{DPL21}, and the OVI Boolean and monotone additive convolutions were provided. 

The main purpose of this paper is to investigate the OVI analogues of multiplicative convolutions for free, Boolean, and monotone independent random variables. We apply the method of \cite{PL19} and \cite{DPL21} to reduce the OVI famework to an $2\times 2$ upper triangular probability space,
 for which we have rich tools in the OV realm to deal with them. For the free case, 
 we introduce the notion of the infinitesimal $T$-transform $\partial T$ (see Definition \ref{defn:inf-T=transform}), and we establish that if self-adjoint elements $x$ and $y$ are infinitesimally free from an operator-valued infinitesimal probability space $(\mathcal{A},\mathcal{B},E,E')$ where $E(x)$ and $E(y)$ are invertible, then for $\|w\|$ small enough, the following relation holds:  
\begin{eqnarray*}
 \partial T_{xy}(w) 
&=& T_x\Big(T_y(w)wT_y(w)^{-1}\Big)\partial T_y(w)+ T_x'(T_y(w)wT_y(w)^{-1})(C)T_y(w)\\ 
&&+\partial T_x\Big(T_y(w)wT_y(w)^{-1}\Big)T_y(w)
\end{eqnarray*}
with
$$
C=-T_y(w)wT_y(w)^{-1}\partial T_y(w)T_y(w)^{-1}+\partial T_y(w)wT_y(w)^{-1}.
$$
where we recall $T_x$ and $T_y$ represent the $T$-transform of $x$ and $y$ respectively, (see the definition in Subsection \ref{Subsection:2.1}). (Note that the OVI $S$-transforms and their related product formula have been investigated in \cite{PL19}.)

Additionally, the $\eta$-transform and $\kappa$-transform play a crucial role in characterizing the corresponding multiplicative convolution for Boolean and monotone independent random variables respectively (see details in Subsection \ref{Subsection:2.1}). We introduce the corresponding infinitesimal transforms $\partial \eta$ and $\partial \kappa$ (see Definition \ref{defn:inf_eta_transform} and Definition \ref{defn:inf_kappa_transform}); furthermore, we provide the Boolean and monotone infinitesimal multiplicative convolution. More precisely, if selfadjoint elements $x$ and $y$ are infinitesimally Boolean independent in $\mathcal{A}$, then the following relation holds for $\|b\|$ small enough: 
$$\partial \eta_{(1+x)(1+y)}(b)=\eta_{1+x}(b)\partial \eta_{1+y}(b)+\partial \eta_{1+x}(b)\eta_{1+y}(b).$$

On the other hand, if we assume $x=x^*$ and $y=y^*$ are elements in $\mathcal{A}$ such that $x-1$ and $y$ are infinitesimally monotone independent in $\mathcal{A},$ then 
$$
\partial \kappa_{yx}(b)=\kappa_x'(\kappa_y(b))(\partial \kappa_y(b))+\partial \kappa_x(\kappa_y(b)) \qquad \text{ for }\|b\| \text{ small enough.}
$$
(Refer to the details in Theorem \ref{OVITfcon}, Theorem \ref{OVIbcon}, and Theorem \ref{OVIMprod}).

In \cite{Dy07}, Dykema introduced the notion of non-crossing linked partitions, and demonstrated how to express moments in terms of $t$-coefficients $\{t_n\}_{n\geq 0}$, which are the coefficients of the (OV) $T$-transform. Subsequently, many properties of $t$-coefficients and non-crossing linked partitions were further studied in \cite{POP-linked-08,N10,EKG21}.  
In Section 5, we introduced the notion of infinitesimal $t$-coefficients $\{t_n'\}_{n\geq 0}$ and extended the properties discussed in \cite{POP-linked-08}. More precisely, we illustrated how to express infinitesimal free cumulants in terms of $t$-coefficients and infinitesimal $t$-coefficients. Furthermore, we provided an equivalent statement of infinitesimal freeness through $\{t_n,t_n'\}_{n\geq 0}$ (see Subsection \ref{Subsect: Multi-t-coeff}). Additionally, in Subsection \ref{Subsect: single-T-coeff}, we presented an application involving complex Wishart matrices. 

In addition to Section 5, we review fundamental knowledge of OV and OVI probability theory in Section 2. Moving to Section 3, for each notion of OVI independence, we introduce the corresponding infinitesimal transforms and demonstrate that these transforms satisfy certain multiplicative properties. In Section 4, we provide a differentiable paths approach to study infinitesimal multiplicative convolutions, as outlined in Theorem \ref{DPconv}.

\tableofcontents


\section{Preliminaries} 

\subsection{Operator-Valued Probability Theory }\label{Subsection:2.1}

$(\mathcal{A},\mathcal{B},E)$ is called an \emph{operator-valued probability space} ( or \emph{OV probability space} for short) if $\mathcal{A}$ is a $C^*$-unital algebra, $\mathcal{B}$ is a $C^*$-unital subalgebra, and $E:\mathcal{A}\to \mathcal{B}$ is a linear $\mathcal{B}$-$\mathcal{B}$ bimodule completely positive map. Given an OV probability space $(\mathcal{A},\mathcal{B},E)$ and a given element $x\in\mathcal{A}$, the \emph{operator-valued distribution of $x$} is the linear map $\nu:\mathcal{B}\langle X\rangle\to\mathcal{B}$ completely determined by 
$$
\nu(Xb_1Xb_2X\cdots Xb_nX)=E(xb_1xb_2x\cdots xb_nx)
$$
where $\mathcal{B}\langle X\rangle$ is the free algebra generated by an indeterminate variable $X$ over $\mathcal{B}$. 
\begin{defn}
Suppose that $(\mathcal{A},\mathcal{B},E)$ is an OV probability space. 
\begin{itemize}
  \item [(1)] 
  The subalgebras $(\mathcal{A}_i)_{i\in I}$ of $\mathcal{A}$ contain $\mathcal{B}$ are \emph{freely independent} over $\mathcal{B}$ if for all $n\in \mathbb{N}$, $a_1,\dots, a_n\in \mathcal{A}$, such that $a_j\in\mathcal{A}_{i_j}$ where $i_1,\dots,i_n\in I$, $i_1\neq \cdots \neq i_n$, and $E(a_j)=0$ for all $j$, then      
  $$
  E(a_1\cdots a_n) = 0.
  $$
  \item [(2)]
  The (possibly non-unital) $\mathcal{B}$-bimodule subalgebras $(\mathcal{A}_i)_{i\in I}$ of $\mathcal{A}$ are \emph{Boolean independent} over $\mathcal{B}$ if for all $n\in \mathbb{N}$, $a_1,\dots, a_n\in \mathcal{A}$ such that $a_j\in\mathcal{A}_{i_j}$ where $i_1,\dots,i_n\in I$, $i_1\neq \cdots \neq i_n$, then we have
  $$
  E(a_1\cdots a_n)=E(a_1)\cdots E(a_n).
  $$
  \item [(3)]
  Assume $I$ is equipped with a linear order $<$. The (possibly non-unital) $\mathcal{B}$-bimodule subalgebras $(\mathcal{A}_i)_{i\in I}$ of $\mathcal{A}$ are \emph{monotone independent} over $\mathcal{B}$ if  
  $$
  E(a_1\cdots a_{j-1}a_ja_{j+1}\cdots a_n) = E(a_1\cdots a_{j-1}E(a_j)a_{j+1}\cdots a_n) 
  $$
  whenever $a_j\in\mathcal{A}_{i_j}$, $i_1,\dots,i_n\in I$ and $i_{j-1}<i_j>i_{j+1}$ where one of the inequalities is eliminated if $j=1$ or $j=n$.
\end{itemize}
\end{defn}
Given an OV probability space $(\mathcal{A},\mathcal{B},E)$, elements $(x_i)_{i \in I}$ of $\mathcal{A}$ are said to be freely independent over $\mathcal{B}$ if the unital algebras $Alg(1,x_i)$, generated by elements $x_i$ over $\mathcal{B}$, are freely independent. Similarly, $(x_i)_{i\in I}$ are said to be Boolean (respectively monotone) independent if $Alg(x_i)$, the non-unital subalgebra generated by $x_i, (i\in I)$ over $\mathcal{B}$ form a Boolean (respectively monotone) independent family.

For a given OV probability space $(\mathcal{A},\mathcal{B},E)$ and $x\in\mathcal{A}$, we write
$$
x=\operatorname{Re}(x)+i\operatorname{Im}(x) \text{ where }\operatorname{Re}(x)=\frac{1}{2}(x+x^*) \text{ and }\operatorname{Im}(x)=\frac{1}{2i}(x-x^*). 
$$
Also, we write $x>0$ if $x$ is positive and invertible, and then the \emph{operator-valued upper half plane} is defined by 
$$
H^+(\mathcal{B})=\{b\in\mathcal{B} \mid \operatorname{Im}(b)>0\}. 
$$
Let $x\in\mathcal{A}$, we define the \emph{Cauchy transform of $x$} by 
$$
G_x(b)=E\left((b-x)^{-1}\right) \text{ for all }b\in\mathcal{B} \text{ with }b-x \text{ is invertible}. 
$$
It is known that $G_x$ is Fr\'{e}chlet analytic and maps $H^+(\mathcal{B})$ into $H^-(\mathcal{B}):=-H^+(\mathcal{B})$. Note that if we consider its matrix amplification $(G_x^{(m)})_{m=1}^{\infty}$ of $G_x$, then $(G_x^{(m)})_{m=1}^\infty$ encodes the operator-valued distribution of $x$ (see \cite{V2000} and \cite{VOI95}). In addition, given $x=x^*\in\mathcal{A}$, we also consider the following transforms: 
\begin{align*}
\psi_x(b)&=E\left((1-bx)^{-1}-1\right);\ &\eta_x(b)&= \widetilde{\psi}_x(b)\left(\psi_x(b)+1\right)^{-1};&  \\
\vartheta_x(b)&= E\left((1-bx)^{-1}bx\right);\ &\kappa_x(b)&=\left(1+\vartheta_x(b)\right)^{-1}\vartheta_x(b);& \\
\varrho_x(b)&= E\left(xb(1-xb)^{-1}\right);\ &\rho_x(b)&= \varrho_x(b)\left(1+\varrho_x(b)\right)^{-1}&
\end{align*} 
where $\widetilde{\psi}_x(b)=b^{-1}\psi_x(b)$. 
All the previous transforms are well-defined and Fr\'{e}chlet analytic on a neighborhood of the origin. Note that $\psi'_x(0)(\cdot)=(\cdot)E(x)$, therefore if we further assume that $E(x)$ is invertible, then $\psi'_x(0)$ is invertible, which implies that $\psi_x$ is invertible around $0$ by the inverse function theorem. Then the \emph{$S$-transform of $x$} is defined by 
$$
S_x(b)=b^{-1} (1+b)\psi_x^{\langle-1\rangle}(b), \text{ for }\|b\| \text{ small enough.} 
$$
Note that the operator-valued version of $S$-transform was first introduced by Voiculescu in \cite{VOI95}, and the previous construction of $S$-transform is based on Dykema's approach in \cite{Dy06}. Moreover, Dykema introduced the notion of the linked partition and use it to describe the reciprocal of 
$S$-transform, which is called the $T$-transform (see \cite{Dy07}).

Suppose that $(\mathcal{A},\mathcal{B},E)$ is an OV probability space, and $x=x^*$ and $y=y^*$ are two elements in $\mathcal{A}$. Then for $\|b\|$ small enough, we have the following formula.  
\begin{itemize}
    \item If $x$ and $y$ are free, and also $E(x)$ and $E(y)$ are both invertible, then
    \begin{equation}\label{Sconvolution}
    S_{xy}(b)=S_y(b)S_x\left(S_y(b)^{-1}bS_y(b)\right)\qquad \text{   (see \cite{Dy06}).}  
    \end{equation} and
    \begin{equation}\label{Tconvolution}
    T_{xy}(b)=T_x\left(T_y(b)bT_y(b)^{-1}\right)T_y(b) \qquad \text{   (see \cite{Dy07}).} 
  \end{equation}
    \item If $x$ and $y$ are Boolean independent, then
    \begin{equation}\label{Bconvolution}
    \eta_{(1+x)(1+y)}(b)=\eta_{1+x}(b)\eta_{1+y}(b)\qquad \text{   (see \cite{POP09}).}
    \end{equation}
    \item If $x-1$ and $y$ are monotone independent, then 
    \begin{eqnarray}\label{Kconvolution}
    \kappa_{yx}(b)&=&(\kappa_x\circ\kappa_y)(b), \\
    \rho_{xy}(b)&=& (\rho_x\circ \rho_y)(b) \qquad \text{   (see \cite{POP08}).}\nonumber  
    \end{eqnarray}
\end{itemize}
\begin{rem}
$(\mathcal{A},\mathcal{B},E)$ is called an \emph{operator-valued Banach probability space} if $\mathcal{A}$ is a unital Banach algebra, $\mathcal{B}$ is a subalgebra of $\mathcal{A}$ that containing $1_{\mathcal{A}}$, and $E:\mathcal{A}\to \mathcal{B}$ is a linear, bounded, $\mathcal{B}$-$\mathcal{B}$ bimodule projection. Note that \eqref{Sconvolution}, \eqref{Bconvolution}, and \eqref{Kconvolution} also hold if $(\mathcal{A},\mathcal{B},E)$ is only an operator-valued Banach non-commutative probability space. 
\end{rem}
\subsection{The Notions of Operator-Valued Infinitesimal Independence}

$\ $

We call $(\mathcal{A},\mathcal{B},E,E')$ an \emph{operator-valued infinitesimal probability space} (or \emph{OVI probability space} for short) if $(\mathcal{A},\mathcal{B},E)$ is an OV probability space, and $E':\mathcal{A}\to \mathcal{B}$ is a linear selfadjoint bounded map such that 
$$
E'(1)=0 \text{ and }E'(b_1ab_2)=b_1E'(a)b_2 \text{ for all }b_1,b_2\in\mathcal{B}, \text{ and }a\in\mathcal{A}. 
$$
Given $x\in \mathcal{A}$, the \emph{operator-valued infinitesimal distribution of $x$} is a pair of linear maps $(\nu,\nu')$ where $\nu$ is the distribution of $x$ and $\nu':\mathcal{B}\langle X\rangle\to \mathcal{B}$ is the map completely determined by 
$$
\nu'(Xb_1Xb_2X\cdots Xb_nX)=E'(xb_1xb_2x\cdots xb_nx). 
$$
\begin{defn}
Suppose $(\mathcal{A},\mathcal{B},E, E')$ is an OVI probability space. 
\begin{itemize}
    \item [(1)]
The sub-algebras $(\mathcal{A}_i)_{i\in I}$ of $\mathcal{A}$ that contain $\mathcal{B}$ are called \emph{infinitesimally
free} with respect to $(E,E')$ if for $i_1,i_2,\dots,i_n\in I$, $i_1\neq i_2 \neq i_3\cdots\neq i_n$, and $a_j\in \mathcal{A}_{i_j}$ with $E(a_{j})=0$ for all $j=1,2,\dots,n$, the following two conditions hold: 
\begin{eqnarray*}
E(a_1\cdots a_n) &=& 0\ \ ; \\ 
E'(a_1\cdots a_n) &=& \sum\limits_{j=1}^n E(a_1\cdots a_{j-1}E'(a_j)a_{j+1}\cdots a_n). 
\end{eqnarray*}
    \item [(2)]
The (possibly non-unital) $\mathcal{B}$-$\mathcal{B}$ bimodule subalgebras $(\mathcal{A}_i)_{i\in I}$ of $\mathcal{A}$ are called \emph{infinitesimally Boolean independent} if for all $n\in\mathbb{N}$ and $a_1,\dots, a_n\in\mathcal{A}$ such that $a_j\in\mathcal{A}_{i_j}$ where $i_1\neq i_2\neq \dots \neq i_n\in I$, we have
\begin{eqnarray*}
E(a_1\cdots a_n) &=& E(a_1)\cdots E(a_n) ; \\
E'(a_1\cdots a_n) &=&\sum_{j=1}^n E(a_1)\cdots E(a_{j-1})E'(a_j)E(a_{j+1})\cdots E(a_{n}). 
\end{eqnarray*}
    \item [(3)]
Assume that $I$ is equipped with a linear order $<$. The (possibly non-unital) $\mathcal{B}$-$\mathcal{B}$ bimodule subalgebras $(\mathcal{A}_i)_{i\in I}$ of $\mathcal{A}$ are called \emph{infinitesimally monotone independent} if 
\begin{eqnarray*}
E(a_1\cdots a_j \cdots a_n) &=& E(a_1\cdots a_{j-1}E(a_j)a_{j+1}\cdots a_n) ;\\
E'(a_1\cdots a_j \cdots a_n) &=&  E(a_1\cdots a_{j-1}E'(a_j)a_{j+1}\cdots a_n) +  E'(a_1\cdots a_{j-1}E(a_j)a_{j+1}\cdots a_n) 
\end{eqnarray*}
whenever $a_j\in\mathcal{A}_{i_j}, i_j\in I$ for all $j$ and $ i_{j-1} < i_j > i_{j+1}$, where one of the inequalities is eliminated if $j=1$ or $j=n$.
\end{itemize}
\end{defn}
For a given OVI probability space $(\mathcal{A},\mathcal{B},E,E')$, the corresponding \emph{upper triangular probability space} $(\widetilde{\mathcal{A}},\widetilde{\mathcal{B}},\widetilde{E})$ consists of 
$$
\widetilde{\mathcal{A}}=\left\{\begin{bmatrix}
a & a' \\ 0 & a 
\end{bmatrix} \Bigg|\ a,a'\in\mathcal{A} \right\}, \widetilde{\mathcal{B}}=\left\{ \begin{bmatrix}
b & b' \\ 0 & b \end{bmatrix} \Bigg|\ b,b'\in\mathcal{B}
  \right\},
$$
and $\widetilde{E}:\widetilde{\mathcal{A}}\to\widetilde{\mathcal{B}}$ is given by
$$
\widetilde{E}\left(\begin{bmatrix}
a & a' \\ 0 & a
\end{bmatrix}\right)=\begin{bmatrix}
E(a) & E'(a)+E(a') \\ 0 & E(a)
\end{bmatrix}.
$$
Consider an OVI probability space $(\mathcal{A},\mathcal{B},E,E')$ and its corresponding upper triangular probability space, the following properties establish the connection between these two spaces (see \cite{DPL21} and \cite{PL19}).  
\begin{prop}\label{UpperProp}
Sub-algebras $(\mathcal{A})_{i\in I}$ that contain $\mathcal{B}$ are infinitesimally free with respect to $(E,E')$ if and only if $(\widetilde{\mathcal{A}}_i)_{i\in I}$ are free with respect to $\widetilde{E}$, where  
$$\widetilde{\mathcal{A}}_i=\left\{ \begin{bmatrix} a & a' \\ 0 & a 
\end{bmatrix} \Bigg |\ a,a'\in\mathcal{A}_i\right\}  \text{ for each }i\in I.$$  
Similarly, $\mathcal{B}$-$\mathcal{B}$ bimodule subalgebras $(\mathcal{A}_i)_{i\in I}$ are infinitesimally Boolean (respectively monotone) independent with respect to $(E,E')$ if and only if $(\widetilde{\mathcal{A}}_i)_{i\in I}$ are Boolean (respectively monotone) independent with respect to $\widetilde{E}$.
\end{prop}
Note that for an OVI probability space $(\mathcal{A},\mathcal{B},E,E')$, the corresponding upper triangular probability space $(\widetilde{\mathcal{A}},\widetilde{\mathcal{B}},\widetilde{E})$ is an operator-valued Banach probability space where we put a norm on $\widetilde{\mathcal{A}}$ by
$$
\left\|A\right\|_{\widetilde{\mathcal{A}}}:=\|a\|_{\mathcal{A}}+\|a'\|_{\mathcal{A}} \qquad \text{ for }A=\begin{bmatrix}
a & a' \\ 0 & a\end{bmatrix}\in \widetilde{\mathcal{A}},
$$
(see \cite{DPL21} and \cite{PL19}). We will omit the subindex in the norm, as it is clear that we use $\|\cdot\|_{\widetilde{\mathcal{A}}}$ on upper triangular matrices and $\|\cdot\|_{\mathcal{A}}$ on $\mathcal{A}$.

Suppose that $(\mathcal{A},\mathcal{B},E,E')$ is an OVI probability space and $x=x^*\in\mathcal{A}$ be given, the \emph{infinitesimal Cauchy transform of $x$} is defined by
$$
g_x(b)=E\left((b-x)^{-1}\right), \text{ for all }b\in\mathcal{B} \text{ with }b-x \text{ is invertible}. 
$$
Then $g_x$ is Fr\'{e}chlet analytic on $H^+(\mathcal{B})$; furthermore, the matrix amplifications $(g_x^{(m)})_{m=1}^\infty$ encodes all possible infinitesimal moments of $x$ (see \cite{PL19}). 

In addition, let $x=x^*\in\mathcal{A}$, the \emph{infinitesimal moment-generating function} $\partial \psi_x$ of $x$ is defined by
\begin{equation}\label{defn:inf_moment_function}
\partial \psi_x(b)=E'((1-bx)^{-1}-1)=\sum\limits_{n\geq 1} E'((bx)^n) \text{ for }\|b\|<\|x\|^{-1}.
\end{equation}
Then if further assume $E(x)$ is invertible, then we define the \emph{infinitesimal $S$-transform} of $x$ by 
$$
\partial S_x(w)=-(\psi_{x}^{\langle -1\rangle})'(w)\left(\partial\psi_x(\psi_x^{\langle -1 \rangle}(w))\right) \text{ for }\|w\| \text{ small enough}.
$$
In \cite{PL19}, the author provide the OVI multiplicative convolution as follows.
\begin{thm}\label{OVIMfcon}
Suppose that $(\mathcal{A},\mathcal{B},E,E')$ is an OVI probability space. Let $x=x^*$ and $y=y^*$ be two infinitesimally freely independent random variables in $\mathcal{A}$ such that $E(x)$ and $E(y)$ are invertible. Then for $\|w\|$ small enough, we have
\begin{eqnarray*}
&& \partial S_{xy}(w) \\ 
&=& S_y(w) S_x'(S_y(w)^{-1}wS_y(w)) \left( S_y(w)^{-1}w\partial S_y(w)-S_y(w)^{-1}\partial S_y(w)S_y(w)^{-1}wS_y(w) \right) \\
 &+& S_y(w) \partial S_x(S_y(w)^{-1}wS_y(w)) + \partial S_y(w) S_x(S_y(w)^{-1}wS_y(w)). 
\end{eqnarray*}
\end{thm}
We note that when $\mathcal{B}=\mathbb{C}$, if self-adjoint elements $x$ and $y$ are infinitesimally free, we obtain
\begin{equation}\label{formula: scalar-S-convolution}
\partial S_{xy}(w)=\partial S_x(w)S_y(w)+S_x(w)\partial S_y(w)  \text{ for }\|w\| \text{ small}.    
\end{equation}

\subsection{Infinitesimal Free Probability} 

The concept of operator-valued infinitesimal freeness was discussed in the preceding subsection, wherein (scalar-valued) infinitesimal freeness serves as a special case when considering $\mathcal{B}=\mathbb{C}$.
In this subsection, our emphasis shifts to the combinatorial facet of scalar-valued infinitesimal free probability theory, offering essential tools for Section \ref{section: inf t-coeff}.

A pair $(\mathcal{A},\varphi)$ is called a \emph{non-commutative probability space} (or \emph{ncps} for short) if $\mathcal{A}$ is a unital algebra over $\mathbb{C}$ and $\varphi:\mathcal{A}\to\mathbb{C}$ is linear with $\varphi(1)=1$. If $(\mathcal{A},\varphi)$ is a ncps with an additional linear functional $\varphi':\mathcal{A}\to\mathbb{C}$ such that $\varphi'(1)=0$, then we call the triple $(\mathcal{A},\varphi,\varphi')$ an \emph{infinitesimal probability space}(or \emph{incps} for short). 

The unital subalgebras $\{\mathcal{A}_i\}_{i\in I}$ are said to be \emph{ freely independent} (or \emph{free} for short) if for all $n\geq 0$, $a_1,\dots,a_n\in \mathcal{A}$ are such that $a_j\in \mathcal{A}_{i_j}$ with $\varphi(a_j)=0$ where $i_1,\dots,i_n\in I$, $i_1\neq \cdots \neq i_n$, then
\begin{eqnarray*}
\varphi(a_1\cdots a_n) =0.
\end{eqnarray*}
Under the same assumption, if we additionally have 
$$
\varphi'(a_1\cdots a_n)= \sum\limits_{j=1}^n \varphi'(a_j)\varphi(a_1\cdots a_{j-1}a_{j+1}\cdots a_n),
$$
then we say $\{\mathcal{A}_i\}_{i\in I}$ are \emph{infinitesimally free}. 

The combinatorial aspect of infinitesimal freeness relies on the notions of free cumulants and infinitesimal free cumulants. To introduce them, we first require the following concept.
A \emph{non-crossing partition} of $[n]:=\{1,2,\cdots,n\}$ is a collection of disjoint subsets $V_1,\dots,V_s$ of $[n]$ (called blocks) which satisfy the following conditions: 
\begin{itemize}
    \item[(1)]  $V_1\cup \cdots \cup V_s=[n]$;
    \item[(2)]  there are no blocks $V_r$ and $V_s$ that $i,k\in V_r$ and $j,l\in V_s$ where $1\leq i<j<k<l\leq n$.
\end{itemize}
The set of all non-crossing partitions is denoted by $NC(n)$. 

To define the free and infinitesimal free cumulants, we introduce the following notations: 
Suppose $\{f_n:\mathcal{A}^n\to\mathbb{C}\}_{n\geq 1}$ and $\{f_n':\mathcal{A}^n\to\mathbb{C}\}_{n\geq 1}$ are two sequences of multi-linear functionals. Given a partition $\pi\in NC(n)$, we define
$$
f_{\pi}(a_1,\dots,a_n)=\prod_{V\in \pi}f_{|V|}((a_1,\dots,a_n)|_V)
$$
where $(a_1,\dots, a_n)|_V=(a_{i_1},\dots, a_{i_s}) $ if  $V=\{i_1<\cdots<i_s\}$.

Furthermore, for a fixed $V\in \pi$,
we define $\partial f_{\pi,V}$ be the map that is equal to $f_{\pi}$ except for the block $V$, where we replace $f_{|V|}$ by $f'_{|V|}$. Then we define $\partial f_{\pi}$ by
$$
\partial f_{\pi}(a_1,\dots,a_n)= \sum_{V\in \pi}\partial f_{\pi,V}(a_1,\dots,a_n).
$$

The \emph{free cumulants} $\{r_n:\mathcal{A}^n\to\mathbb{C}\}_{n\geq 1}$ are defined inductively in terms of moments via 
$$
\varphi(a_1\cdots a_n)=\sum_{\pi\in NC(n)}r_{\pi}(a_1,\dots,a_n).
$$
Additionally, the \emph{infinitesimal free cumulants} $\{r_n':\mathcal{A}^n\to\mathbb{C}\}_{n\geq 1}$ are defined inductively in terms of moments and infinitesimal moments via  
$$
\varphi'(a_1,\dots,a_n)=\sum_{\pi\in NC(n)}\partial r_{\pi}(a_1,\dots,a_n). 
$$
$\{r_n,r_n'\}_{n\geq 1}$ can be used to characterize infinitesimal freeness, as stated in the following result. 
\begin{thm}[\cite{FN10}]
Suppose that $(\mathcal{A},\varphi,\varphi')$ is an infinitesimal probability space, and $\mathcal{A}_i$ is a unital subalgebra of $\mathcal{A}$ for each $i\in I$. Then $(\mathcal{A}_i)_{i\in I}$ are infinitesimally free if and only if for each $s\geq 2$ and $i_1,\dots,i_s\in I$ which are not all equal, and for $a_1\in \mathcal{A}_{i_1},\dots,a_s\in \mathcal{A}_{i_s}$, we have $r_n(a_1,\dots,a_s)=r_n'(a_1,\dots,a_s)=0.$
\end{thm}

\subsection{Non-crossing Linked Partitions and $t$-coefficients }

In \cite{Dy07}, Dykema introduces the notion of the linked partitions and $t$-coefficients, offering a new combinatorial aspect for the study of free multiplicative convolutions. Under the assumption $\mathcal{B}=\mathbb{C}$, 
Popa \cite{POP-linked-08} investigates the properties of $t$-coefficients in the multi-variable case. In \cite{MN10}, Mastnak and Nica concentrate on the single variable case, and they establish a connection between the linked partitions and the non-crossing partitions (also see \cite{N10}). In this subsection, we will provide a brief review of their works.   

Given $n\in \mathbb{N}$, a \emph{non-crossing linked partition} of $[n]$ is a collection of subsets $V_1,\dots,V_k$ that satisfies same properties $(1)$ and $(2)$ as non-crossing partitions, but we allow for any $r,s\in [k]$, $|V_r\cap V_s|\leq 1$. If $V_r\cap V_s=\{j\},$ then $j$ is the minimal element of only one of the blocks $V_r$ and $V_s$. The set of all non-crossing linked partitions is denoted by $NCL(n)$. In addition for a given $\pi\in NCL(n)$, we set 
$$s(\pi)=\{k\in [n] \mid k\text{ is not a minimal element in any blocks of } \pi \}.$$

Let $(\mathcal{A},\varphi)$ be an ncps, and define $\hat{\mathcal{A}}=\{a\in\mathcal{A}\mid \varphi(a)\neq 0\}.$ The \emph{$t$-coefficients} $\{t_n:\mathcal{A}\times\hat{\mathcal{A}}^{n}\to\mathbb{C}\}_{n\geq 0}$ is a sequence of maps that is defined inductively by the following equation:
$$
\varphi(a_1\cdots a_n)=\sum_{\pi\in NCL(n)} t_{\pi}(a_1,\dots,a_n)
$$
where 
$$
t_{\pi}(a_1,\dots,a_n)=\prod_{\substack{V\in \pi; \\ V=\{i_1,\dots,i_s\} }}t_{s-1}(a_{i_1},\dots,a_{i_s})\prod_{k\in s(\pi)}t_0(a_k). 
$$ 
For $\pi\in NCL(n)$, we say $i,j\in [n]$ are connected if there exist blocks $V_1,\dots,V_s\in \pi$ such that $i\in V_1$ and $j\in V_s$ with $V_k\cap V_{k+1}\neq \phi$ for all $k\in [s-1]$. We denote it by $i\sim_{\pi} j.$ 
\begin{nota}
Given $\pi\in NC(n)$, we define $$
\langle\pi\rangle=\{\sigma\in NCL(n)\mid i\sim_{\pi}j \text{ if and only if }i\sim_{\sigma}j \text{ for all }i,j\in [n]\}.$$
\end{nota}
In \cite{POP-linked-08}, Popa showed that the free cumulants can be expressed in terms of $t$-coefficients. Moreover, he provided a characterization of freeness by using $t$-coefficients. We review these results below.
\begin{prop}\label{prop: t-coefficient}
For any $n$, and $a_1,\dots,a_n\in\hat{\mathcal{A}}$, we have
\begin{equation}
r_n(a_1,\dots,a_n)=\sum_{\pi\in \langle 1_n\rangle} t_{\pi}(a_1,\dots,a_n).
\end{equation}
\end{prop}
\begin{prop}\label{Prop: vanishing of mixed t-coeff}
Unital subalgebras $\{\mathcal{A}_i\}_{i\in I}$ are free if and only if for all $a_k\in \mathcal{A}_{i_k}$ with $\varphi(a_k)=1$ for all $k$, and $i_1,\dots,i_n\in I$, we have
$$
t_{n-1}(a_1,\dots,a_n)=0
$$
whenever there are $r,s\in [n]$ with $i_r\neq i_s$. 
\end{prop}

Considering the single variable case. Let's simplify our notations as follows. For a given $a\in\hat{\mathcal{A}}$, we write $r_n(a)=r_n(a,\dots,a)$ and $t_n(a)=t_n(a,\cdots,a).$ Note that for any $\pi\in \langle 1_n\rangle$, $|s(\pi)|=n-\#(\pi)$ where $\#(\pi)$ is the number of blocks in $\pi$. Combining Proposition \ref{prop: t-coefficient} and \cite[Lemma 6.2]{MN10}, for all $n\geq 2$, we have the following result:
\begin{equation}\label{eqn: single-variable - K vs t}
r_n(a)= \sum_{\pi\in NC(n-1)}\Big(\prod_{V\in \pi}t_{|V|}(a)\Big)t_0(a)^{n-\#(\pi)}.
\end{equation}
In addition, following \cite[Remark 1.3]{POP-linked-08}, we note that for a given $a\in \hat{\mathcal{A}}$ and $c\in\mathbb{C}$, we have
$t_n(ca)=ct_n(a)$ for all $n\geq 0$, and hence, $T_{ca}(z) = cT_a(z).$  

\section{OVI Multiplicative Convolutions}

In this section, we introduce several transforms which play key role to describe the (operator-valued ) infinitesimal multiplicative convolutions for free, Boolean, and monotone respectively. We will apply Proposition \ref{UpperProp} to reduce our problems to the operator-valued case, and then we can deduce our results by using formulas \eqref{Sconvolution}, \eqref{Bconvolution}, and \eqref{Kconvolution}. 

Suppose that $(\mathcal{A},\mathcal{B},E,E')$ is an OVI probability space. Let $x=x^*$ be an element in $\mathcal{A}$, then we rcall that the moment generating function $\psi_x$ is defined by 
$\psi_x(b)=E((1-bx)^{-1}-1) \text{ for all }\|b\|<\|x\|^{-1}.$  

If we consider $X=\begin{bmatrix}
x & 0 \\ 0 & x 
\end{bmatrix}$, then the moment generating function $\Psi_X$ of $X$ can be defined similarly. Moreover, it can be expressed as
$$
\Psi_{X}\left(\begin{bmatrix}
b & c \\
0 & b
\end{bmatrix}\right) = \begin{bmatrix}
\psi_x(b) & \psi_x'(b)(c)+\partial \psi_x(b) \\
0 & \psi_x(b)
\end{bmatrix}
$$
for any $b,c\in\mathcal{B}$ with $\left\|\begin{bmatrix}
b & c \\ 0 & b
\end{bmatrix}\right\|:=\|b\|+\|c\|<\|x\|^{-1}$, 
where $\psi'_x$ represents the Fr{\'e}chet derivative of $\psi_x$ and $\partial\psi_x$ is the infinitesimal moment generating function of $x$, defined as \eqref{defn:inf_moment_function} (see \cite[Section 4.3]{PL19}).

\subsection{Free case}

Suppose $(\mathcal{A},\mathcal{B},E,E')$ is an OVI probability space. We let $x=x^*\in \mathcal{A}$ and further assume that $E(x)$ is invertible. Then let $X=\begin{bmatrix}
    x & 0 \\ 0 & x
\end{bmatrix}$, we consider the $T$-transform of $X$ as follows 
\begin{eqnarray*}
    &&T_X\left(\begin{bmatrix}
        w & v \\ 0 & w
    \end{bmatrix}\right) \\
    &=& S_X\left(\begin{bmatrix}
        w & v \\ 0 & w
    \end{bmatrix}\right)^{-1} = \begin{bmatrix}
        S_x(w) & S_x'(w)(v)+\partial S_x(w) \\
0 & S_x(w)
    \end{bmatrix}^{-1} \\ 
&=& \begin{bmatrix}
    S_x(w)^{-1} & -S_x(w)^{-1} \big(S_x'(w)(v)+\partial S_x(w)\big) S_x(w)^{-1} \\ 0 & S_x(w)^{-1}
\end{bmatrix} \\
&=& \begin{bmatrix}
    T_x(w) & \partial T_x(w)-S_x(w)^{-1}S_x'(w)(v)S_x(w)^{-1} \\ 0 & T_x(w)
\end{bmatrix}
= \begin{bmatrix}
    S_x(w)^{-1} & \partial T_x(w)+T_x'(w)(v) \\ 0 & S_x(w)^{-1}
\end{bmatrix}
\end{eqnarray*}
where 
\begin{equation}\label{formula: OVI-T-transform}
\partial T_x(w)=-T_x(w) \partial S_x(w) T_x(w).
\end{equation}
\begin{defn}\label{defn:inf-T=transform}
The \eqref{formula: OVI-T-transform} is called the \emph{infinitesimal $T$-transform }of $x$. 
\end{defn}
\begin{thm}\label{OVITfcon}
Suppose that $(\mathcal{A},\mathcal{B},E,E')$ is an OVI probability space. Let $x=x^*$ and $y=y^*$ be two infinitesimally freely independent random variables in $\mathcal{A}$ such that $E(x)$ and $E(y)$ are invertible. Then for $\|w\|$ small enough, we have
\begin{eqnarray*}
\partial T_{xy}(w) 
&=& T_x\Big(T_y(w)wT_y(w)^{-1}\Big)\partial T_y(w)+ T_x'(T_y(w)wT_y(w)^{-1})(C)T_y(w)\\ 
&&+\partial T_x\Big(T_y(w)wT_y(w)^{-1}\Big)T_y(w)
\end{eqnarray*}
where 
$$
C=-T_y(w)wT_y(w)^{-1}\partial T_y(w)T_y(w)^{-1}+\partial T_y(w)wT_y(w)^{-1}.
$$
\end{thm}
\begin{proof}
Following Proposition \ref{UpperProp}, $X=\begin{bmatrix}
    x & 0 \\ 0 & x
\end{bmatrix}$ and $Y=\begin{bmatrix}
    y & 0 \\ 0 & y
\end{bmatrix}$ are free. Then for all $\|w\|$ small, we have
\begin{eqnarray*}
&&T_{XY}\left(\begin{bmatrix}
    w & 0 \\ 0 & w
\end{bmatrix}\right) \\
&=& T_X\left(\begin{bmatrix} 
    T_y(w) & \partial T_y(w) \\ 0 & T_y(w)
\end{bmatrix}\begin{bmatrix}
    w & 0 \\ 0 & w
\end{bmatrix}\begin{bmatrix}
    T_y(w)^{-1} & -T_y(w)^{-1}\partial T_y(w)T_y(w)^{-1} \\ 0 & T_y(w)^{-1}
\end{bmatrix}\right)\begin{bmatrix}
    T_y(w) & \partial T_y(w) \\ 0 & T_y(w)\end{bmatrix} 
\end{eqnarray*}
via \eqref{Tconvolution}. We note that the $(1,2)$-entry of $T_{XY}\left(\begin{bmatrix}
    w & 0 \\ 0 & w
\end{bmatrix}\right) $ is $\partial T_{xy}(w)$; thus, we shall compute the $(1,2)$-entry of the right hand side of the above equation. 
Note that 
\begin{eqnarray*}
&&\begin{bmatrix}
    T_y(w) & \partial T_y(w) \\ 0 & T_y(w)
\end{bmatrix}\begin{bmatrix}
    w & 0 \\ 0 & w
\end{bmatrix}\begin{bmatrix}
    T_y(w)^{-1} & -T_y(w)^{-1}\partial T_y(w)T_y(w)^{-1} \\ 0 & T_y(w)^{-1}
\end{bmatrix} \\
&=& \begin{bmatrix}
    T_y(w)wT_y(w)^{-1} & -T_y(w)wT_y(w)^{-1}\partial T_y(w)T_y(w)^{-1}+\partial T_y(w)wT_y(w)^{-1} \\
    0 & T_y(w)wT_y(w)^{-1}\end{bmatrix} \\
    &=&\begin{bmatrix}
    T_y(w)wT_y(w)^{-1} & C \\ 0 & T_y(w)wT_y(w)^{-1}
    \end{bmatrix}
\end{eqnarray*}
where $C= -T_y(w)wT_y(w)^{-1}\partial T_y(w)T_y(w)^{-1}+\partial T_y(w)wT_y(w)^{-1}$. 

Therefore,
\begin{eqnarray*}
&&T_X\left(\begin{bmatrix}
    T_y(w)wT_y(w)^{-1} & C \\ 0 & T_y(w)wT_y(w)^{-1}
    \end{bmatrix}\right)\begin{bmatrix}
    T_y(w) & \partial T_y(w) \\ 0 & T_y(w)\end{bmatrix} \\
    &=& \begin{bmatrix}
        T_x\Big(T_y(w)^{-1}wT_y(w)\Big) & T_x'\Big(T_y(w)^{-1}wT_y(w)\Big)(C)+\partial T_x\Big(T_y(w)^{-1}wT_y(w)\Big) \\ 0 & T_x\Big(T_y(w)^{-1}wT_y(w)\Big)
    \end{bmatrix}\begin{bmatrix}
    T_y(w) & \partial T_y(w) \\ 0 & T_y(w)\end{bmatrix}.
\end{eqnarray*}
Thus, by computing the $(1,2)$-entry of the last equation, we complete the proof. 
\end{proof}
\begin{rem}
We note that Theorem \ref{OVITfcon} can be also proved by combining Theorem \ref{Sconvolution} and applying 
$$
\partial S_x(w)=-T_x(w)^{-1}\partial T_x(w) T_x(w)^{-1}.
$$
In addition, when $\mathcal{B}=\mathbb{C}$, the infinitesimal $T$-transform of a given element $x$ is 
$$
\partial T_x(w)=\frac{-\partial S_x(w)}{S_x(w)^2} = -T_x(w)\partial S_x(w) T_x(w) \text{ for }|w| \text{ small. }
$$
Moreover, if $x=x^*$ and $y=y^*$ are infinitesimally free, then 
\begin{equation}\label{eqn: scalar-version t-convolution}
\partial T_{xy}(w)=\partial T_x(w) T_y(w)+T_x(w) \partial T_y(w)
\text{ for } |w| \text{ small. }\end{equation}
\end{rem}
\subsection{Boolean Case}
Suppose that $(\mathcal{A},\mathcal{B},E,E')$ is an OVI probability space. Let $x=x^*$ be an element in $\mathcal{A}$, and then we also let $
X=\begin{bmatrix}
x & 0 \\ 0 & x 
\end{bmatrix}.$ Let us compute the $B$-transform of $X$ as the following. 
\begin{lem}\label{B_X formula}
For $b$ and $c\in\mathcal{B}$ small enough, we have 
$$
\eta_x\left(\begin{bmatrix}
b & c \\
0 & b
\end{bmatrix}\right)=\begin{bmatrix}
\eta_x(b) & \eta_x'(b)(c)+\partial \eta_x(b)\\
0 & \eta_x(b)
\end{bmatrix}
$$
where 
$$
\partial \eta_x(b)=(\psi_x(b)+1)^{-1}\partial \psi_x(b)(\psi_x(b)+1)^{-1}.
$$
\end{lem}
\begin{proof}
Observe that
\begin{eqnarray*}
&&\left(\Psi_X\left(\begin{bmatrix}
b & c \\ 0 & b
\end{bmatrix}\right)+I\right)^{-1}\\
&=&  \left(\begin{bmatrix}
\psi_x(b) & \psi_x'(b)(c)+\partial \psi_x(b) \\
0 & \psi_x(b)
\end{bmatrix}+\begin{bmatrix}
1 & 0 \\
0 & 1
\end{bmatrix}\right)^{-1} \\
&=& \begin{bmatrix}
\psi_x(b)+1 & \psi_x'(b)(c)+\partial \psi_x(b) \\
0 & \psi_x(b)+1
\end{bmatrix} ^{-1} \\
&=& \begin{bmatrix}
(\psi_x(b)+1)^{-1} & -(\psi_x(b)+1)^{-1}(\psi_x'(b)(c)+\partial \psi_x(b) )(\psi_x(b)+1)^{-1}\\
0 & (\psi_x(b)+1)^{-1}
\end{bmatrix}.
\end{eqnarray*}
Then we have
\begin{eqnarray*}
\eta_x\left(\begin{bmatrix}
b & c \\
0 & b
\end{bmatrix}\right)&=& \Psi_X\left(\begin{bmatrix}
b & c \\ 0 & b
\end{bmatrix}\right)\left(\Psi_X\left(\begin{bmatrix}
b & c \\ 0 & b
\end{bmatrix}\right)+I\right)^{-1} \\
&=& \begin{bmatrix}
\psi_x(b)(\psi_x(b)+1)^{-1} & C_x(b,c) \\
0 & \psi_x(b)(\psi_x(b)+1)^{-1}
\end{bmatrix}
\end{eqnarray*}
where
\begin{eqnarray*}
C_x(b,c)&=& -\psi_x(b)\left[(\psi_x(b)+1)^{-1}(\psi_x'(b)(c)+\partial \psi_x(b) )(\psi_x(b)+1)^{-1}\right]\\ &&+\left[\psi_x'(b)(c)+\partial \psi_x(b)\right](\psi_x(b)+1)^{-1}.
\end{eqnarray*}
Note that
\begin{eqnarray*}
&&-\psi_x(b)\left[(\psi_x(b)+1)^{-1}(\psi_x'(b)(c)+\partial \psi_x(b) )(\psi_x(b)+1)^{-1}\right] \\ &&+\left[\psi_x'(b)(c)+\partial \psi_x(b)\right](\psi_x(b)+1)^{-1} \\
&=& \psi_x(b)((\psi_x+1)^{-1})'(b)(c)+(\psi_x)'(b)(c)(\psi_x(b)+1)^{-1} \\
&& +\left(1-\psi_x(b)(\psi_x(b)+1)^{-1}\right)\partial \psi_x(b)(\psi_x(b)+1)^{-1} \\
&=& \left(\psi_x(\psi_x+1)^{-1}\right)'(b)(c)+\left(1-\psi_x(b)(\psi_x(b)+1)^{-1}\right)\partial \psi_x(b)(\psi_x(b)+1)^{-1} \\
&=& \eta_x'(b)(c)+\left(1-\psi_x(b)(\psi_x(b)+1)^{-1}\right)\partial \psi_x(b)(\psi_x(b)+1)^{-1} \\
&=& \eta_x'(b)(c)+(\psi_x(b)+1)^{-1}\partial \psi_x(b)(\psi_x(b)+1)^{-1}.
\end{eqnarray*}
\end{proof}
\begin{defn}\label{defn:inf_eta_transform}
The map $\partial \eta_x$ in Lemma \ref{B_X formula} is called the (operator-valued) \emph{infinitesimal $\eta$-transform of x}. 
\end{defn}
\begin{thm}\label{OVIbcon}
Suppose that $(\mathcal{A},\mathcal{B},E,E')$ is an OVI probability space. Let $x=x^*$ and $y=y^*$ be two infinitesimally Boolean independent random variables in $\mathcal{A}$. Then for $b\in \mathcal{B}$ small enough, we have
$$\partial \eta_{(1+x)(1+y)}(b)=\eta_{1+x}(b)\partial \eta_{1+y}(b)+\partial \eta_{1+x}(b)\eta_{1+y}(b).$$
\end{thm}
\begin{proof}
By Proposition \ref{UpperProp}, we have $X=\begin{bmatrix}
x & 0 \\ 0 & x
\end{bmatrix}$ and $Y=\begin{bmatrix}
y & 0 \\ 0 & y
\end{bmatrix}$ are Boolean independent. Hence, by \eqref{Bconvolution} we obtain 
\begin{equation}\label{prodBeqn}
\eta_{(1+X)(1+Y)}\left(\begin{bmatrix}
b & c \\
0 & b
\end{bmatrix}\right)= \eta_{1+X}\left(\begin{bmatrix}
b & c \\
0 & b
\end{bmatrix}\right)\eta_{1+Y}\left(\begin{bmatrix}
b & c \\
0 & b
\end{bmatrix}\right).
\end{equation}
The left hand side of \eqref{prodBeqn} is
$$
\begin{bmatrix}
\eta_{(1+x)(1+y)}(b) & \eta_{(1+x)(1+y)}'(b)(c)+\partial \eta_{(1+x)(1+y)}(b)\\
0 & \eta_{(1+x)(1+y)}(b)
\end{bmatrix}.
$$
In addition, the right hand side of \eqref{prodBeqn} is
\begin{eqnarray*}
&&\begin{bmatrix}
\eta_{1+x}(b) & \eta_{1+x}'(b)(c)+\partial \eta_{1+x}(b) \\
0 & \eta_{1+x}(b)
\end{bmatrix} \begin{bmatrix}
\eta_{1+y}(b) & \eta_{1+y}'(b)(c)+\partial \eta_{1+y}(b) \\
0 & \eta_{1+y}(b)
\end{bmatrix} \\
&=& \begin{bmatrix}
\eta_{1+x}(b)\eta_{1+y}(b) & \eta_{1+x}(b)\left(\eta_{1+y}'(b)(c)+\partial \eta_{1+y}(b)\right)+\left(\eta_{1+x}'(b)(c)+\partial \eta_{1+x}(b)\right)\eta_{1+y}(b)\\
0 & \eta_{1+x}(b)\eta_{1+y}(b)
\end{bmatrix}  \\
&=& \begin{bmatrix}
\eta_{1+x}(b)\eta_{1+y}(b) & (\eta_{1+x}\eta_{1+y})'(b)(c)+\left(\eta_{1+x}(b)\partial \eta_{1+y}(b)+\partial \eta_{1+x}(b)\eta_{1+y}(b)\right)\\
0 & \eta_{1+x}(b)\eta_{1+y}(b) &
\end{bmatrix}.
\end{eqnarray*}
Comparing the $(1,2)$-entry on both side of \eqref{prodBeqn}, we obtain the desired result.
\end{proof}
\begin{rem}
When $\mathcal{B}=\mathbb{C}$, we get the formula of scalar version of infinitesimal $B$-transform and the infinitesimal multiplicative Boolean convolution as follows. 

For a given random variable $x\in\mathcal{A}$, the infinitesimal $\eta$-transform of $x$ is given by
$$
\partial \eta_x(z)=\frac{\partial \psi_x(z)}{(z\psi_x(z)+1)^2}\qquad  \text{ for }|z| \text{ small.}
$$
Furthermore, if $x$ and $y$ are infinitesimally Boolean independent, then 
$$
\partial \eta_{(1+x)(1+y)}(z)=\eta_{1+x}(z)\partial \eta_{1+y}(z)+\partial \eta_{1+x}(z)\eta_{1+y}(z).
$$
\end{rem}
\subsection{Monotone Case}
Given an OVI probability space $(\mathcal{A},\mathcal{B},E,E')$ and $x=x^*$ in $\mathcal{A}$. We let $X=\begin{bmatrix}
x & 0 \\ 0 & x
\end{bmatrix}$, and then the following transforms of $X$ are all well-defined in an neighborhood of the origin.   
\begin{eqnarray*}
\vartheta_X\left(\begin{bmatrix}
b & c \\ 0 & b 
\end{bmatrix}\right)&=&\widetilde{E}\left(\left(I-\begin{bmatrix}
b & c \\ 0 & b
\end{bmatrix}\begin{bmatrix}
x & 0 \\ 0 & x 
\end{bmatrix}\right)^{-1}\begin{bmatrix}
b & c \\ 0 & b 
\end{bmatrix}\begin{bmatrix}
x & 0 \\ 0 & x
\end{bmatrix}\right), \\
\kappa_X\left(\begin{bmatrix}
b & c \\ 0 & b
\end{bmatrix}\right)&=&\left(1+\vartheta_X\left(\begin{bmatrix}
b & c \\ 0 & b
\end{bmatrix}\right)\right)^{-1}\vartheta_X\left(\begin{bmatrix}
b & c \\ 0 & b
\end{bmatrix}\right),   \\
\varrho_X\left(\begin{bmatrix}
b & c \\ 0 & b
\end{bmatrix}\right) &=& \widetilde{E}\left(
\begin{bmatrix}
x & 0 \\ 0 & x
\end{bmatrix}
\begin{bmatrix}
b & c \\ 0 & b 
\end{bmatrix}\left(I-\begin{bmatrix}
x & 0 \\ 0 & x
\end{bmatrix}
\begin{bmatrix}
b & c \\ 0 & b 
\end{bmatrix}\right)^{-1}\right), \\
\rho_X\left(\begin{bmatrix}
b & c \\ 0 & b
\end{bmatrix}\right)&=&\varrho_X\left(\begin{bmatrix}
b & c \\ 0 & b
\end{bmatrix}\right)\left(1+\varrho_X\left(\begin{bmatrix}
b & c \\ 0 & b
\end{bmatrix}\right)\right)^{-1}.  
\end{eqnarray*}
First, observe that for $b$ and $c\in\mathcal{B}$ small enough, we have 
\begin{eqnarray*}
&& \vartheta_X\left(\begin{bmatrix}
 b & c \\ 0 & b
 \end{bmatrix}\right) \\
 &=& \widetilde{E}\left(\begin{bmatrix}
 1-bx & -cx \\ 0 & 1-bx
 \end{bmatrix}^{-1}\begin{bmatrix}
 bx & cx \\ 0 & bx
 \end{bmatrix}\right) \\
&=& \widetilde{E} \left(\begin{bmatrix}
(1-bx)^{-1}bx & (1-bx)^{-1}cx[1+(1-bx)^{-1}bx]\\
0 & (1-bx)^{-1}bx
\end{bmatrix}\right) \\
&=& \begin{bmatrix}
E\left((1-bx)^{-1}bx\right) & E\left((1-bx)^{-1}cx[1+(1-bx)^{-1}bx]\right)+E'\left((1-bx)^{-1}bx\right)\\
0 & E\left((1-bx)^{-1}bx\right)
\end{bmatrix} \\
&=& 
\begin{bmatrix}
\vartheta_x(b) & \vartheta_x'(b)(c)+\partial \vartheta_x(b) \\
0 & \vartheta_x(b)
\end{bmatrix}
\end{eqnarray*}
where 
$\partial \vartheta_x(b)=E'\left((1-bx)^{-1}bx\right).$ Note that $ b^{-1}\vartheta_x(b)=\varrho_x(b)b^{-1} $, 
we can also easily obtain 
$$
\varrho_X\left(\begin{bmatrix}
b & c \\ 0 & b
\end{bmatrix}\right)=\begin{bmatrix}
\varrho_x(b) & \varrho_x(b)'(c)+\partial \varrho_x(b) \\
0 & \varrho_x(b)
\end{bmatrix}
$$
where $$\partial \varrho_x(b)=E'(xb(1-xb)^{-1}).$$ 

Now, let us compute the $\kappa$-transform of $X$.  
\begin{lem}\label{kappa_Xformula}
$$
\kappa_X\left(\begin{bmatrix}
b & c \\ 0 & b
\end{bmatrix}\right)=\begin{bmatrix}
\kappa_x(b) & \kappa_x'(b)(c)+\partial \kappa_x(b)\\
0 & \kappa_x(b)
\end{bmatrix}  \qquad \text{ for }\|b\|,\|c\| \text{ small enough,}
$$
where
$$
\partial \kappa_x(b)=(1+\vartheta_x(b))^{-1}\partial \vartheta_x(b)(1+\vartheta_x(b))^{-1}.
$$
\end{lem}
\begin{proof}
 Note that 
\begin{eqnarray*}
&&\kappa_X\left(\begin{bmatrix}
b & c \\ 0 & b
\end{bmatrix}\right) \\
&=& \left(\begin{bmatrix}
1+\vartheta_x(b) & \vartheta_x'(b)(c)+\partial \vartheta_x(b) \\
0 & 1+\vartheta_x(b)
\end{bmatrix}\right)^{-1}\begin{bmatrix}
\vartheta_x(b) & \vartheta_x'(b)(c)+\partial \vartheta_x(b) \\
0 & \vartheta_x(b)
\end{bmatrix} \\
&=& 
\begin{bmatrix}
(1+\vartheta_x(b))^{-1} & -(1+\vartheta_x(b))^{-1}(\vartheta_x'(b)(c)+\partial \vartheta_x(b))(1+\vartheta_x(b))^{-1} \\
0 & (1+\vartheta_x(b))^{-1}
\end{bmatrix}\begin{bmatrix}
\vartheta_x(b) & \vartheta_x'(b)(c)+\partial \vartheta_x(b) \\
0 & \vartheta_x(b)
\end{bmatrix} \\
&=& 
\begin{bmatrix}
(1+\vartheta_x(b))^{-1}\vartheta_x(b) & C_x(b,c) \\
0 & (1+\vartheta_x(b))^{-1}\vartheta_x(b)
\end{bmatrix}
\end{eqnarray*}
where
$$
C_x(b,c)=(1+\vartheta_x(b))^{-1}(\vartheta_x'(b)(c)+\partial \vartheta_x(b))-(1+\vartheta_x(b))^{-1}(\vartheta_x'(b)(c)+\partial \vartheta_x(b))(1+\vartheta_x(b))^{-1}\vartheta_x(b).
$$
Observe that 
\begin{eqnarray*}
&&C_x(b,c) \\
&=& (1+\vartheta_x(b))^{-1}\vartheta_x'(b)(c)\left(1-(1+\vartheta_x(b))^{-1}\vartheta_x(b)\right)+(1+\vartheta_x(b))^{-1}\partial \vartheta_x(b)\left(1-(1+\vartheta_x(b))^{-1}\vartheta_x(b)\right) \\
&=& \kappa'_x(b)(c)+(1+\vartheta_x(b))^{-1}\partial \vartheta_x(b)\left(1-(1+\vartheta_x(b))^{-1}\vartheta_x(b)\right)\\
&=& \kappa'_x(b)(c)+(1+\vartheta_x(b))^{-1}\partial \vartheta_x(b)
(1+\vartheta_x(b))^{-1}
\end{eqnarray*}
Thus, we complete the proof. 
\end{proof}
\begin{defn}\label{defn:inf_kappa_transform}
The map $\partial \kappa_x$ in Lemma \ref{kappa_Xformula} is called the (operator-valued) \emph{infinitesimal $\kappa$-transform of x}. 
\end{defn}
\begin{thm}\label{OVIMprod}
Suppose that $(\mathcal{A},\mathcal{B},E,E')$ is an OVI probability space. Let $x=x^*$ and $y=y^*$ be elements in $\mathcal{A}$ that $x-1$ and $y$ are infinitesimally monotone independent random variables in $\mathcal{A}$. Then 
$$
\partial \kappa_{yx}(b)=\kappa_x'(\kappa_y(b))(\partial \kappa_y(b))+\partial \kappa_x(\kappa_y(b)) \qquad \text{ for }\|b\| \text{ small enough.}
$$
\end{thm}
\begin{proof}
By Proposition \ref{UpperProp}, $X-1$ and $Y$ are monotone where $X=\begin{bmatrix}
x & 0 \\ 0 & x
\end{bmatrix}$ and $Y=\begin{bmatrix}
y & 0 \\ 0 & y
\end{bmatrix}$. 
By \eqref{Kconvolution}, we get
\begin{equation}\label{prodKeqn}
\kappa_{YX}\left(\begin{bmatrix}
b & c \\ 0 & b
\end{bmatrix}\right)=(\kappa_X\circ \kappa_Y)\left(\begin{bmatrix}
b & c \\ 0 & b
\end{bmatrix}\right).
\end{equation}
The left hand side of \eqref{prodKeqn} is 
$$
\begin{bmatrix}
\kappa_{yx}(b) & \kappa_{yx}'(b)(c)+\partial \kappa_{yx}(b) \\
0 & \kappa_{yx}(b)
\end{bmatrix}.
$$
Note that the right hand side of \eqref{prodKeqn} is 
\begin{eqnarray*}
(\kappa_X\circ \kappa_Y)\left(\begin{bmatrix}
b & c \\ 0 & b
\end{bmatrix}\right)
&=& \kappa_X\left(\begin{bmatrix}
\kappa_y(b) & \kappa_y'(b)(c)+\partial \kappa_y(b) \\
0 & \kappa_y(b)
\end{bmatrix}\right) \\
&=& \begin{bmatrix}
\kappa_x(\kappa_y(b)) & \kappa_x'(\kappa_y(b))(\kappa_y'(b)(c)+\partial \kappa_y(b))+\partial \kappa_x(\kappa_y(b))\\
0 & \kappa_x(\kappa_y(b))
\end{bmatrix} \\
&=& \begin{bmatrix}
(\kappa_x\circ\kappa_y)(b) & (\kappa_x\circ\kappa_y)'(b)(c)+\kappa_x'(\kappa_y(b))(\partial \kappa_y(b))+\partial \kappa_x(\kappa_y(b))\\
0 & (\kappa_x\circ \kappa_y)(b)
\end{bmatrix}.
\end{eqnarray*}
Thus, if we compare $(1,2)$-entry of \eqref{prodKeqn}, then we obtain
$$
\partial \kappa_{yx}(b)=\kappa_x'(\kappa_y(b))(\partial \kappa_y(b))+\partial \kappa_x(\kappa_y(b)).
$$
\end{proof}

The following results can be done by using the similar proof in Lemma \ref{kappa_Xformula} and Theorem \ref{OVIMprod}, we omit the proof and state the results as follows. 
\begin{cor}
Let $x=x^*\in\mathcal{A}$ be an element from an OVI probability space $(\mathcal{A},\mathcal{B},E,E')$, we have
\begin{equation*}
\rho_X\left(\begin{bmatrix}
b & c \\ 0 & b
\end{bmatrix}\right)=\begin{bmatrix}
\rho_x(b) & \rho_x'(b)(c)+\partial \rho_x(b) \\
0 & \rho_x(b)
\end{bmatrix}  \qquad \text{ for }\|b\|,\|c\| \text{ small enough,}
\end{equation*}
where
$$
\partial \rho_x(b)=(1+\varrho_x(b))^{-1}\varrho_x(b)(1+\varrho_x(b))^{-1}.
$$
Moreover, if $x=x^*$ and $y=y^*$ are elements in $\mathcal{A}$ such that $x-1$ and $y$ are infinitesimally monotone independent, then
$$
\partial \rho_{xy}(b)=\rho'_x(\rho_y(b))(\partial \rho_y(b))+\partial \rho_x(\rho_y(b)).
$$
\end{cor}
\begin{rem}
When $\mathcal{B}=\mathbb{C}$, we get 
\begin{align*}
\partial \vartheta_x(z)&= E'((1-zx)^{-1}zx);\ &\partial \kappa_x(z)&= \frac{\partial \vartheta_x(z)}{(1+\vartheta_x(z))^2};&  \\
\partial \varrho_x(z)&= E'(xz(1-xz)^{-1});\ &\partial \rho_x(z)&= \frac{\partial \varrho_x(z)}{(1+\varrho_x(z))^2},& 
\end{align*}
which are all well-defined in an neighborhood of the origin. Moreover, if $x-1$ and $y$ are infinitesimally monotone independent, then 
\begin{eqnarray*}
\partial \kappa_{yx}(z)&=&\kappa_x'(\kappa_y(z))\partial \kappa_y(z)+\partial \kappa_x(\kappa_y(z)) \\
\partial \rho_{xy}(z)&=&\rho_x'(\rho_y(z))\partial \rho_y(z)+\partial \rho_x(\rho_y(z)).
\end{eqnarray*}
\end{rem}
\section{Differentiable Paths}

In this section we apply differentiable paths to compute the infinitesimal multiplicative convolutions, which was obtained in Section 3. 

Let us review the framework of differentiable paths. Suppose that $\mathcal{B}$ is a unital $C^*$-algebra, and let $\mathcal{B}\langle X \rangle$ be the $*$-algebra of non-commutative polynomials over $\mathcal{B}$ such that $\mathcal{B}$ and $X$ are algebraically free with $X=X^*.$ 
We denote $\Sigma$ to be the space of all $\mathcal{B}$-valued distributions; that is, $\mu\in \Sigma$ if and only if $\mu$ is a linear $\mathcal{B}$-bimodule completely positive map $\mu:\mathcal{B}\langle X\rangle \to \mathcal{B}$ with $\mu(1)=1$. A $\mathcal{B}$-distribution $\mu$ is called \emph{exponentially bounded} if there is $M>0$ such that for all $b_1,\dots,b_n\in\mathcal{B}$, we have
$$
\| Xb_1Xb_2\cdots b_nX \| \leq M^{n+1} \|b_1\|\|b_2\|\cdots \|b_n\|. 
$$
and $\Sigma_0$ be the space of all $\mathcal{B}$-valued exponentially bounded distributions. 

Note that for each $\mu\in\Sigma_0$, $\{G^{(n)}_{\mu}\}_{n=1}^{\infty}$ encodes $\mu$ (see \cite{VOI95}) where
\begin{equation*} 
G^{(n)}_{\mu}(b):=\mu\otimes 1_n[(b-X\otimes 1_n)^{-1}]=\sum\limits_{n=0}^{\infty} \mu\otimes1_n[ (b^{-1}\cdot X\otimes 1_n)^nb^{-1}],\ \forall\ b\in H^+(M_n(\mathcal{B})).
\end{equation*}
In addition, there exists a OV probability space $(\mathcal{A},\mathcal{B},E)$ and $x=x^*\in\mathcal{A}$ such that $\mu_x=\mu$ (see \cite{PoVi}). 

From the analytic point of view, $G_{\mu}^{(n)}$ on $H^+(M_n(\mathcal{B}))$ essentially has the same behavior of $G_{\mu}:=G_{\mu}^{(1)}$ on $H^+(\mathcal{B}),$ so we will focus on the ground level. 

A path $\mu_t:[0,1]\to \Sigma_0$ is said to be \emph{differentiable} if for each $n\in\mathbb{N}$, $G^{(n)}_{\mu_t}$ is differentiable on $H^+(M_n(\mathcal{B}))$ in the sense that for each $t\in [0,1]$ there is a map $Y^{(n)}_t:H^+(M_n(\mathcal{B}))\to M_n(\mathcal{B})$ such that 
$$
\Big\|\frac{G^{(n)}_{\mu_{t+h}}(b)-G^{(n)}_{\mu_t}(b)}{h} -Y^{(n)}_t(b) \Big\| \longrightarrow 0 \text{ as } h\rightarrow 0 \text{ for all } b\in H^+(M_n(\mathcal{B})).
$$

We denote the ground level limit $Y_t^{(1)}$ by $\partial G_{\mu_t}$, and the Fr{\'e}chet derivative of $G_{\mu_t}$ will be denoted by $G'_{\mu_t}$. That is, 
$$
\partial G_{\mu_t}(b)=\lim_{h\to0}\!\frac{G_{\mu_{t+h}}(b)-G_{\mu_t}(b)}{h}
\text{ and }
G_{\mu_t}'(b)(\cdot)=\lim_{h\to0}\!\frac{G_{\mu_t}(b+h\cdot)-G_{\mu_t}(b)}{h}.
$$
Note that for a given $\mu\in\Sigma_0$, we can also consider many transforms of $\mu$ such as $\psi_{\mu},\ S_{\mu},\ \eta_{\mu}$ and so on. 

Suppose that $(\mathcal{A},\mathcal{B},E)$ is an OV probability space. Recall that we said $x=x^*\in\mathcal{A}$ follows $\mathcal{B}$-valued distribution $\mu\in \Sigma_0$ (we denote it by $x\sim \mu$) if for $n\in\mathbb{N}$ and $b_1,\dots,b_n\in\mathcal{B}$, we have
$$
E(xb_1x\cdots xb_nx) = \mu(Xb_1X\cdots Xb_nX).
$$
Let us recall the following notations. 
Suppose that $x$ and $y$ are self-adjoint elements in $\mathcal{A}$ with $x\sim\mu$ and $y\sim\nu$ for some $\mu,\nu\in\Sigma_0$.
If $x\geq 0$ and $y$ in $\mathcal{A}$ are free over $\mathcal{B}$, then we denote the $\mathcal{B}$-valued distribution of $xy$ by $\mu\boxtimes \nu$.
Suppose that $x-1$ and $y-1$ are Boolean independent over $\mathcal{B}$, then the $\mathcal{B}$-valued distribution of $xy$ is denoted to be $\mu\utimes \nu$. If $x-1$ and $y$ are monotone independent over $\mathcal{B}$, then we denote the $\mathcal{B}$-valued distribution of $xy$ by $\mu\circlearrowright\nu$.
\begin{thm}\label{DPconv}
Suppose that $\mu_t$ and $\nu_t$ are differentiable paths, then on an neighborhood of the origin, one has
\begin{eqnarray}
\label{Feqn}
\partial S_{\mu_t\boxtimes \nu_t} &=& \partial S_{\nu_t}S_{\mu_t}(S_{\nu_t}^{-1}\cdot Id\cdot S_{\nu_t}) + S_{\nu_t} \partial S_{\mu_t}(S_{\nu_t}^{-1}\cdot Id\cdot S_{\nu_t}) \nonumber \\
&& + S_{\nu_t}S_{\mu_t}'\left(S_{\nu_t}^{-1}\cdot Id\cdot S_{\nu_t}\right)\left(S_{\nu_t}^{-1}\cdot Id\cdot \partial S_{\nu_t}-S_{\nu_t}^{-1}\cdot \partial S_{\nu_t}\cdot S_{\nu_t}^{-1}\cdot Id\cdot S_{\nu_t}\right); 
\end{eqnarray}
\begin{eqnarray}\label{F2eqn}
\partial T_{\mu_t\boxtimes \nu_t} &=& T_{\mu_t}(T_{\nu_t}\cdot Id\cdot T_{\nu_t}^{-1})\partial T_{\nu_t} + \partial T_{\mu_t}(T_{\nu_t}\cdot Id\cdot T_{\nu_t}^{-1})T_{\nu_t} \nonumber\\
    &&+ T_{\mu_t}'(T_{\nu_t}\cdot Id\cdot T_{\nu_t}^{-1})(\partial T_{\nu_t}\cdot Id\cdot T_{\nu_t} -T_{\nu_t}\cdot Id\cdot T_{\nu_t}^{-1}\cdot \partial T_{\nu}\cdot T_{\nu_t}^{-1} )T_{\nu_t}; \end{eqnarray}
\begin{eqnarray}
\label{Beqn}
\partial \eta_{\mu_t\utimes \nu_t}&=&\eta_{\mu_t}\partial \eta_{\nu_t}+\partial \eta_{\mu_t}\eta_{\nu_t}; \\
\label{Meqn1}
\partial \kappa_{\nu_t\circlearrowright\mu_t} &=& \kappa_{\mu_t}'(\kappa_{\nu_t})(\partial \kappa_{\nu_t})+\partial \kappa_{\mu_t}(\kappa_{\nu_t}); \\
\label{Meqn2}
\partial \rho_{\mu_t\circlearrowright\nu_t} &=& \rho_{\mu_t}'(\rho_{\nu_t})(\partial\rho_{\nu_t})+\partial \rho_{\mu_t}(\rho_{\nu_t}). 
\end{eqnarray}
\end{thm}
\begin{proof}
Note that for $t>0$ and $b\in\mathcal{B}$ small enough, we have
\begin{eqnarray*}
\eta_{\mu_t\utimes \nu_t}(b)&=&\eta_{\mu_t}(b)\eta_{\nu_t}(b); \\
\kappa_{\nu_t\circlearrowright\mu_t}(b) &=& (\kappa_{\mu_t}\circ \kappa_{\nu_t})(b); \\
\rho_{\mu_t\circlearrowright\nu_t}(b) &=& (\rho_{\mu_t}\circ \rho_{\nu_t})(b). 
\end{eqnarray*}
Thus, we obtain \eqref{Beqn}, \eqref{Meqn1}, and \eqref{Meqn2} immediately by differentiating above equations with respect to $t$. 

The proof of \ref{F2eqn} is similar to $\ref{Feqn}$. We will only demonstrate \eqref{Feqn}, observe that for $t>0$ and $w\in \mathcal{B}$ small enough, 
\begin{equation}\label{Sproduct}
    S_{\mu_t\boxtimes \nu_t}(w) = S_{\nu_t}(w) S_{\mu_t}\left(S_{\nu_t}(w)^{-1}wS_{\nu_t}(w) \right).
\end{equation}
By differentiating \eqref{Sproduct} with respect to $t$, we have 
\begin{eqnarray*}
&&\partial S_{\mu_t\boxtimes \nu_t}  \\
&=& \partial S_{\nu_t}(w) S_{\mu_t}\left(S_{\nu_t}(w)^{-1}wS_{\nu_t}(w) \right) + S_{\nu_t}(w) \partial \left[ S_{\mu_t}\left(S_{\nu_t}(w)^{-1}wS_{\nu_t}(w) \right) \right].
\end{eqnarray*}
Note that 
\begin{eqnarray*}
 &&\partial \left[ S_{\mu_t}\left(S_{\nu_t}(w)^{-1}wS_{\nu_t}(w) \right) \right]  \\
 &=& \partial S_{\nu_t}\left(S_{\nu_t}(w)^{-1}wS_{\nu_t}(w)\right) + S_{\nu_t}'\left(S_{\nu_t}(w)^{-1}wS_{\nu_t}(w)\right) \partial \left[ S_{\nu_t}(w)^{-1}wS_{\nu_t}(w) \right] \\
 &=& \partial S_{\nu_t}\left(S_{\nu_t}(w)^{-1}wS_{\nu_t}(w)\right) \\
 && + S_{\nu_t}'\left(S_{\nu_t}(w)^{-1}wS_{\nu_t}(w)\right) \left( S_{\nu_t}(w)^{-1}w\partial S_{\nu_t}(w)-S_{\nu_t}(w)^{-1}\partial S_{\nu_t}(w)S_{\nu_t}(w)^{-1}wS_{\nu_t}(w)\right).
\end{eqnarray*}
Hence, we obtain \eqref{Feqn}. 
\end{proof}

\section{Infinitesimal $t$-coefficients }\label{section: inf t-coeff}

In this section, we will introduce the notion of infinitesimal $t$-coefficients $\{t_n'\}_{n\geq 0}$ and study their properties. Moreover, we provide an application in random matrix theory. Specifically, we will compute the limit infinitesimal law of the product of complex Wishart matrices in terms of $T$-transform and infinitesimal $T$-transform.  

\subsection{Multi-variable case}\label{Subsect: Multi-t-coeff}
Suppose $(\mathcal{A},\varphi,\varphi')$ is an incps. We recall that  $\hat{\mathcal{A}}=\{a\in\mathcal{A}\mid \varphi(a)\neq 0\}$, and for a given $\pi\in NCL(n)$, let $s(\pi)$ be the set of all $k\in [n]$ such that there are no blocks of $\pi$ whose minimal element is $k$. 
\begin{defn}\label{defn: inf t-coeff}
The \emph{infinitesimal $t$-coefficients} $\{t_n':\mathcal{A}\times(\hat{\mathcal{A}})^{n}\to \mathbb{C}\}_{n\geq 0}$ is a sequence of maps given via the following recurrence: 
$$
\varphi'(a_1\cdots a_n)=\sum_{\pi\in NCL(n)} \partial t_{\pi}(a_1,\dots,a_n)
$$
where 
\begin{eqnarray*}
\partial t_{\pi}(a_1,\dots,a_n)&=&\sum_{\substack{V\in\pi; \\ V=\{i_1,\dots,i_s\} }} t'_{s-1}(a_{i_1},\dots,a_{i_s})\prod_{\substack{W\in \pi, W\neq V; \\ W=\{j_1,\dots,j_r\}}}t_{r-1}(a_{j_1},\dots,a_{j_{r}})\prod_{k\in s(\pi)}t_0(a_k) \\
&&+ \prod_{\substack{V\in\pi; \\ V=\{i_1,\dots,i_s\} }} t_{s-1}(a_{i_1},\dots,a_{i_s}) \sum_{k\in s(\pi)} t_0'(a_k)\prod_{k'\in s(\pi);k'\neq k}t_0(a_{k'}).
\end{eqnarray*}
\end{defn}
Let $(\widetilde{\mathcal{A}},\widetilde{\varphi},\widetilde{\mathbb{C}})$ be the upper triangular probability space induced by $(\mathcal{A},\varphi,\varphi')$. Note that the commutativity of $\widetilde{\mathbb{C}}$ inherits many properties in ncps $(\mathcal{A},\varphi)$. First, let $a_1,\dots,a_n\in \hat{\mathcal{A}}$ such that for all $j\in [n]$, we set 
$A_j=\begin{bmatrix}
    a_j & 0 \\ 0 & a_j
\end{bmatrix}$ for each $j\in [n]$. Note that $\widetilde{\varphi}(A_j)=\begin{bmatrix}
    \varphi(a_j) & \varphi'(a_j) \\ 0 & \varphi(a_j)
\end{bmatrix}$ is invertible for all $j\in [n]$. 

Then the $\widetilde{t}$-coefficients can be defined via the similar equation 
$$
\widetilde{\varphi}(A_1\cdots A_n)=\sum_{\pi\in NCL(n)}\widetilde{t}_{\pi}(A_1,\dots,A_n) 
$$
where 
$$
\widetilde{t}_{\pi}(A_1,\dots,A_n)=\prod_{\substack{V\in \pi; \\ V=\{i_1,\dots,i_s\} }}\widetilde{t}_{s-1}(A_{i_1},\dots,A_{i_s})\prod_{k\in s(\pi)}\widetilde{t}_0(A_k).
$$
Following the definition of infinitesimal $t$-coefficients and induction, it is easy to see that
\begin{equation}\label{formula: matrix-valued of t_s}
\widetilde{t}_{s-1}(A_{i_1},\dots,A_{i_s}) = \begin{bmatrix}
    t_{s-1}(a_{i_1},\dots,a_{i_s}) & t_{s-1}'(a_{i_1},\dots,a_{i_s})\\ 0 & t_{s-1}(a_{i_1},\dots,a_{i_s})
\end{bmatrix}.    
\end{equation}
Due to the commutativity of $\widetilde{\mathbb{C}}$, the analogues arguments of Proposition \ref{prop: t-coefficient} and Proposition \ref{Prop: vanishing of mixed t-coeff} hold; therefore, we have the following results immediately.  
\begin{prop}\label{prop: kappa' vs t'}
Suppose $(\mathcal{A},\varphi,\varphi')$ is an incps, and $a_1,\dots,a_n\in\hat{\mathcal{A}}$, then we have
\begin{equation}\label{formula: kappa'-t'}
r_n'(a_1,\dots,a_n) = \sum_{\pi\in \langle 1_n\rangle} \partial t_{\pi}(a_1,\dots,a_n).  
\end{equation}
\end{prop}
\begin{proof}
Let $A_j=\begin{bmatrix}
    a_j & 0 \\ 0 & a_j
\end{bmatrix}$ for all $j=1,\dots,n$. Note that \eqref{formula: kappa'-t'} can be deduced by comparing the $(1,2)$-entries of the following equation 
\begin{equation*}
\widetilde{r}_n(A_1,\dots,A_n)=\sum_{\pi\in \langle 1_n\rangle} \widetilde{t}_{\pi}(A_1,\dots,A_n).
\end{equation*}
\end{proof}
\begin{prop}
Unital subalgebras $\{A_i\}_{i\in I}$ are infinitesimally free if and only if for all $a_k\in \mathcal{A}_{i_k}$ for all $k$ and $i_1,\dots,i_n\in I$, we have
$$
t_{n-1}(a_1,\dots,a_n)=t_{n-1}'(a_1,\dots,a_n)=0
$$
whenever there exist $r,s\in [n]$ with $i_r\neq i_s$.  
\end{prop}
\begin{proof}
The proof follows directly from the analogue statement of Proposition \ref{Prop: vanishing of mixed t-coeff} and \eqref{formula: matrix-valued of t_s}. 
\end{proof}

\subsection{Single variable case and application to random matrix theory}\label{Subsect: single-T-coeff}

In this subsection, we focus on the single variable case. To simplify notations, for a given $a\in\hat{\mathcal{A}}$, we write $r_n(a)=r_n(a,\dots,a), r_n'(a)=r_n'(a,\dots,a) , t_n(a)=t_n(a,\cdots,a),$ and $t_n'(a)=t_n'(a,\dots,a)$. 

Let $a\in\hat{\mathcal{A}}$. Due to the commutativity of $\widetilde{\mathbb{C}}$, we have 
analogous formula of \eqref{eqn: single-variable - K vs t} as follows.
\begin{equation}\label{eqn: Upp-kappa-t-formula}
\widetilde{r}_n(A)=\sum_{\pi\in NC(n-1)}\Big(\prod_{V\in \pi}\widetilde{t}_{|V|}(A)\Big)\widetilde{t}_0(A)^{n-\#(\pi)} \text{ for all }n\geq 2
\end{equation}
where $A=\begin{bmatrix}
    a & 0 \\ 0 & a
\end{bmatrix}.$ 
By comparing the $(1,2)$-entry of \eqref{eqn: Upp-kappa-t-formula}, we obtain for all $n\geq 2$,
\begin{eqnarray}
r'_n(a) &=&\sum_{\pi\in NC(n-1)} \Big(\sum_{V\in \pi}t'_{|V|}(a)\prod_{W\in \pi; W\neq V}t_{|W|}(a)\cdot t_0(a)^{n-\#(\pi)} \\
&&\qquad \qquad+\prod_{V\in \pi}t_{|V|}(a)\cdot (n-\#(\pi)) \cdot t_0'(a)\cdot t_0(a)^{n-\#(\pi)-1}\Big). \nonumber
\end{eqnarray}
In addition, for any $c\in\mathbb{C}$ and $a\in\hat{\mathcal{A}}$, the property $\widetilde{t}_n(cA) = c\widetilde{t}_n(A)$ holds for each $n\geq 0$ and $A=diag(a,a)$ by the commutativity of $\widetilde{\mathbb{C}}$. Then we combine \eqref{formula: matrix-valued of t_s} to deduce that $t_n'(ca)=ct_n'(a).$  

In \cite{mingo19}, Mingo studied the limit infinitesimal distribution of complex Wishart matrices, and he showed that independent complex Wishart matrices are asymptotically infinitesimally free. We will show how to find the limit infinitesimal law of the product of two independent complex Wishart matrices.  

Let us first recall that $X_N$ is said to be a complex Wishart matrix if $X_N=\frac{1}{N}G^*G$ where $G$ is a $M\times N$ Gaussian random matrix with $N(0,1)$ entries. If we assume that
$$
\lim\limits_{N\to \infty}\frac{M}{N} = c, \text{\ and also\ }\lim\limits_{N\to \infty}(M-Nc) = c', 
$$
then there is an infinitesimal distribution for $X_N$ in the sense that there is a infinitesimal probability space $(\mathcal{A},\varphi,\varphi')$ and $x\in\mathcal{A}$ such that
\begin{eqnarray*}
\varphi(x^k)&=&\lim\limits_{N\to \infty} E(tr(X_N^k)); \\
\varphi'(x^k)&=&\lim\limits_{N\to\infty} N[E(tr(X_N^k))-\varphi(x^k)].
\end{eqnarray*}
Mingo showed that $r_n(x)=c$ and $r'_n(x)=c'$ for all $n\geq 1$ (see \cite{mingo19}). We present how to find $t_n(x)$ and $t'_n(x)$ as the following proposition. 

\begin{prop}\label{prop: inf Wishart t-coeff}
$$t_n(x) = \begin{cases}
    c & \text{ if }n=0 \\
    1  & \text{ if }n=1 \\
    0 & \text{ if }n\geq 2.
\end{cases} \text{ and }t'_n(x) = \begin{cases}
    c'& \text{ if }n=0 \\
    0 & \text{ if }n\geq 1,
\end{cases}
$$
which is equivalent to $$T_x(z)=c+z \text{ and } 
\partial T_x(z)=c'. 
$$
\end{prop}
\begin{proof}
We first consider $a=c^{-1}x$, then $t_0(a)=\varphi(a)=r_1(a)=1$. Therefore,  
\begin{eqnarray*}
r_n(a)&=& \sum_{\pi\in NC(n-1)}\Big(\prod_{V\in \pi}t_{|V|}(a)\Big); \\
r'_n(a)&=&\sum_{\pi\in NC(n-1)} \Big(\sum_{V\in \pi}t'_{|V|}(a)\prod_{W\in \pi; W\neq V}t_{|W|}(a)+\prod_{V\in \pi}t_{|V|}(a)\cdot (n-\#(\pi)) \cdot t_0'(a)\Big).
\end{eqnarray*}
We shall show 
$$t_n(a) = \begin{cases}
    1 & \text{ if }n=0 \\
    \frac{1}{c}  & \text{ if }n=1 \\
    0 & \text{ if }n\geq 2.
\end{cases} \text{ and }t'_n(a) = \begin{cases}
    \frac{c'}{c} & \text{ if }n=0 \\
    0 & \text{ if }n\geq 1,
\end{cases}
$$
Then, the final result can be obtained immediately by applying $t_n(x)=ct_n(a)$ and $t_n'(x)=ct_n'(a)$. 

We note that 
$$ r_n(a)= \frac{1}{c^n}r_n(x)=\frac{1}{c^{n-1}}  \text{ for all } n.$$ 
Thus, $t_1(a)=r_2(a)=1/c$. We shall prove $t_n(a)=0$ for all $n\geq 2$ by induction. First of all, $ r_3(a) = t_2(a) + t_1(a)^2 = t_2(a) $ implies that $t_2(a)=r_3(a)-t_1(a)^2=1/c^2-1/c^2=0.$ 

Suppose that $t_{m}(a)=0$ for all $2\leq m\leq k-1$.  We note that
$$
r_{k+1}(a)=t_{k}(a)+\sum_{\substack{\pi\in NC(k); \\ \pi\neq 1_{k},0_k}}\prod_{V\in \pi}t_{|V|}(a)+t_1(a)^{k}.  
$$
Since $t_{m}(a)=0$ for all $2\leq m\leq k-1$, $\prod_{V\in \pi}t_{|V|}(a)=0$ for all $\pi\in NC(k)\setminus\{1_k,0_k\}$.  
Therefore, 
$$
t_k(a)=r_{k+1}(a)-t_1(a)^k=\frac{1}{c^k}-\frac{1}{c^k}=0. 
$$
As a result, we conclude that $t_n(a)=0$ for all $n$ by induction. 

To compute $t_n'(a)$, we first note that
$
r'_n(a)=\frac{c'}{c^n}
$ for all $n$ and also $t_0'(a)=r_1'(a)=\frac{c'}{c}$. Note that $r_2'(a)=t_1'(a)+t_1(a)t_0'(a)$, then we have
$$
t_1'(a)=r_2'(a)-t_1(a)t'_0(a) = \frac{c'}{c^2} - \frac{1}{c}\frac{c'}{c} = 0. 
$$
For $n\geq 2$, we observe that for a given $\pi\in NC(n-1)$, 
\begin{eqnarray*}
\prod_{V\in \pi}t_{|V|}(a)\cdot (n-1) \cdot t_0'(a) = \begin{cases}
    t_1(a)^{n-1}\cdot (n-\#(\pi))\cdot t_0'(a) = \frac{c'}{c^n} & \text{ if } \pi=0_{n-1} \\
    0 & \text{ otherwise}. 
\end{cases}
\end{eqnarray*}
In addition, 
\begin{eqnarray*}
\sum_{\pi\in NC(n-1)}\sum_{V\in \pi}t'_{|V|}(a)\prod_{W\in \pi; W\neq V}t_{|W|}(a) = \sum_{k=0}^{n-2}\tbinom{n-1}{k} t_{n-1-k}'(a)\cdot t_1^k(a) =  \sum_{k=0}^{n-2}\tbinom{n-1}{k} t_{n-1-k}'(a)\cdot \frac{1}{c^k}. 
\end{eqnarray*}
Thus, for all $n\geq 2$, we obtain 
\begin{eqnarray*}
   \sum_{k=0}^{n-2}\tbinom{n-1}{k}t_{n-1-k}'(a)\frac{1}{c^k}+\frac{c'}{c^n} = \frac{c'}{c^n}  
\text { which implies that } 
   \sum_{k=0}^{n-2}\tbinom{n-1}{k}t_{n-1-k}'(a)\frac{1}{c^k} = 0.
\end{eqnarray*}

By substituting $s=n-1-k$ and performing some simplifications, we obtain the following equation: 
\begin{equation}\label{eqn:t'-equation}
\sum_{s=0}^{n-1}\tbinom{n-1}{s}t_s'(a)c^{s+1}=c' \text{ for all }n\geq 2.
\end{equation}
We shall show that $t_n'(a)= 0$ for all $n\geq 2$ by induction. Applying \eqref{eqn:t'-equation} with $n=3$, we have  
$$
\sum_{s=0}^2 \binom{2}{s}t_s'(a)c^{s+1} =c' \text{ implies } t_2'(a) = \frac{1}{c^3}\Big(c'-\binom{2}{0}\frac{c'}{c}\cdot c+\binom{2}{1}\cdot 0\cdot c^2\Big) = 0.
$$
Suppose that $t_{k}'(a)=0$ for all $2\leq k\leq m-1$. 
Applying \eqref{eqn:t'-equation} with $n=m+1$, we have
$$
\sum_{s=0}^m\binom{m}{s}t'_s(a)c^{s+1} = c',
\text{ which implies } 
\binom{m}{0}c'+\sum_{s=1}^{m-1}\binom{m}{s}\cdot 0 \cdot c^{s+1} + t_m'(a)c^{m+1}=c'.
$$
Hence, we obtain $t_m'(a)=0$. By induction, we complete the proof. 
\end{proof}
To study the limit law of product of independent complex Wishart matrices, let $y$ be an infinitesimal free copy of $x$. We deduce the $T$-transform and the infinitesimal $T$-transform of $xy$ by combining Proposition \ref{prop: inf Wishart t-coeff} and  \eqref{eqn: scalar-version t-convolution} as below. 
\begin{thm}
$$T_{xy}(z) = c^2+2cz+z^2 \text{ and }\partial T_{xy}(z)=2c'(c+z).$$
\end{thm}
\begin{proof}
By Proposition \ref{prop: inf Wishart t-coeff}, we first note
that $T_x(z)=T_y(z)=c+z$ and 
$$
\partial T_x(z)=\partial T_y(z)=c'.
$$
Therefore, we have 
$$T_{ab}(z)=T_a(z)T_b(z)=\big(c+z\big)^2= c^2+2cz+z^2.$$ 
In addition, 
\begin{eqnarray*}
    \partial T_{xy}(z)= \partial T_x(z) \cdot T_y(z) + T_x(z) \cdot \partial T_y(z)
    =2 c'(c+z).
\end{eqnarray*}
\end{proof}

\section*{Acknowledgement}

The author would like to express their gratitude to Mihai Popa for the insightful discussions that led to the consideration of non-crossing linked partitions and t-coefficients. The author also extend their thanks to Nicolas Gilliers for the valuable discussions. Furthermore, the authors wish to extend their sincere appreciation to the anonymous reviewers for their insightful comments and constructive suggestions, which have significantly enhanced the quality of this manuscript.
\bibliographystyle{abbrv}

\end{document}